\documentclass{amsart}

\usepackage{enumerate}

\newcommand{\IR}{\mathbb R}
\newcommand{\IZ}{\mathbb Z}
\newcommand{\IQ}{\mathbb Q}
\newcommand{\IN}{\mathbb N}
\newcommand{\pr}{\mathrm{pr}}
\newcommand{\Ra}{\Rightarrow}
\newcommand{\w}{\omega}

\newtheorem{theorem}{Theorem}
\newtheorem{lemma}{Lemma}
\newtheorem{corollary}{Corollary}
\newtheorem{claim}{Claim}
\newtheorem{problem}{Problem}
\newtheorem{proposition}{Proposition}
\theoremstyle{definition}
\newtheorem{example}{Example}
\newtheorem{remark}{Remark}

\title[The continuity of additive and convex functions ...]{The continuity of additive and convex functions,\\ which are upper bounded on non-flat continua in $\IR^n$}
\author{Taras Banakh, Eliza Jab\l o\'nska, Wojciech Jab\l o\'nski}
\address{T. Banakh: Ivan Franko University of Lviv (Ukraine) and Jan Kochanowski University in Kielce (Poland)}
\email{t.o.banakh@gmail.com}
\address{E. Jab\l o\'nska: Department of Discrete Mathematics, Rzesz\'ow University of Technology,
Powstanc\'ow Warszawy~12, 35-959 Rzesz\'ow, Poland}
\email{elizapie@prz.edu.pl}
\address{W. Jab\l o\'{n}ski: Department of Mathematical Modeling, Rzesz\'ow University of Technology,
Powstanc\'ow War\-sza\-wy~12, 35-959 Rzesz\'ow, Poland}
\email{wojtek@prz.edu.pl}

\begin{document}

\begin{abstract} We prove that for a continuum $K\subset \IR^n$ the sum $K^{+n}$ of $n$ copies of $K$ has non-empty interior in $\IR^n$ if and only if $K$ is not flat in the sense that the affine hull of $K$ coincides with $\IR^n$. Moreover, if $K$ is locally connected and each non-empty open subset of $K$ is not flat, then for any (analytic) non-meager subset $A\subset K$ the sum $A^{+n}$ of $n$ copies of $A$ is not meager in $\IR^n$ (and then the sum $A^{+2n}$ of $2n$ copies of the analytic set $A$ has non-empty interior in $\IR^n$ and the set $(A-A)^{+n}$ is a neighborhood of zero in $\IR^n$). This implies that a mid-convex function $f:D\to\IR$, defined on an open convex subset $D\subset\IR^n$ is continuous if it is upper bounded on some non-flat continuum in $D$ or on a non-meager analytic subset of a locally connected nowhere flat subset of $D$.
\end{abstract}
\maketitle

Let $X$ be a linear topological space over the field of real numbers. A function $f:X\to\IR$ is called {\em additive} if $f(x+y)=f(x)+f(y)$ for all $x,y\in X$.

A function $f:D\to\IR$ defined on a convex subset $D\subset X$ is called {\em mid-convex} if $f\big(\tfrac{x+y}{2}\big)\leq\frac{f(x)+f(y)}2$ for all $x,y\in D$.

Many classical results concerning additive or mid-convex functions state that the  boundedness of such functions on ``sufficiently large" sets implies their continuity. That is why Ger and Kuczma~\cite{GerKuczma} introduced the following three families of sets in $X$:
\begin{itemize}
\item the family $\mathcal A(X)$ of all subsets $T\subset X$ such that any mid-convex function $f:D\to\IR$ defined on a~convex open subset $D\subset X$ containing $T$ is continuous if $\sup f(T)<\infty$;
\item the family $\mathcal B(X)$ of all subsets $T\subset X$ such that any additive function $f:X\to\IR$ with $\sup f(T)<\infty$ is continuous;
\item the family $\mathcal C(X)$ of all subsets $T\subset X$ such that any additive function $f:X\to\IR$ is continuous if the set $f(T)$ in bounded in $\IR$.
\end{itemize}
It is clear that
\begin{equation}\label{0x}
\mathcal A(X)\subset\mathcal B(X)\subset\mathcal C(X).
\end{equation}
By the example of Erd\H os \cite{Erdos} (discussed in \cite{GerKuczma}) the classes $\mathcal B(X)$ and $\mathcal C(X)$ are not equal even if $X=\IR^n$, $n\in\IN$. On the other hand, Ger and Kominek \cite{GerKominek} proved that $\mathcal A(X)=\mathcal B(X)$ for any Baire topological vector space $X$. In particular, $\mathcal A(\IR^n)=\mathcal B(\IR^n)$ for every $n\in\IN$ (cf. \cite{MEK}).

There are lots of papers devoted to the problem of recognizing sets in the families $\mathcal A(X)$, $\mathcal B(X)$ or $\mathcal C(X)$, see e.g. \cite{BJ}, \cite{Ger}, \cite{Kurepa}, \cite{Mehdi}, \cite{Ostrowski}.  The classical results concerning mid-convex functions (namely, Bernstein-Doetsch Theorem \cite{BD} and its generalizations, see e.g.~\cite{TTZ}) imply that a subset $T$ with non-empty interior in a topological vector space $X$ belongs to the families $\mathcal A(X)\subset\mathcal B(X)\subset\mathcal C(X)$. By (the proofs) of  Lemma 9.2.1 and Theorem 9.2.5 in \cite{Kuczma2}, a subset $T$ of a~topological vector space $X$ belongs to a family $\mathcal K\in\{\mathcal A(X),\mathcal B(X),\mathcal C(X)\}$ if and only if for some $n\in\IN$ its $n$-fold sum $$T^{+n}=\underbrace{T+\dots+T}_{n \text{\rm\, times}}$$ belongs to the family $\mathcal K$.
Combining these two facts, we obtain the following well-known folklore result.

\begin{theorem}\label{tx}
A subset $T$ of a topological vector space $X$ belongs to the families $\mathcal A(X)\subset\mathcal B(X)\subset\mathcal C(X)$ if for some $n\in\IN$ its $n$-fold sum $T^{+n}$ has non-empty interior in $X$.
\end{theorem}

Theorem~\ref{tx} has been used many times to show that various ``thin'' sets actually belong to the class $\mathcal A(X)$, $\mathcal B(X)$ or $\mathcal C(X)$. In this respect let us mention the following result of Ger \cite{Ger1}.

\begin{theorem}[Ger]\label{Ger}
Let $I\subset\IR$ be a nontrivial interval, $n\geq2$ and let $\varphi:I\to\mathbb{R}^n$ be a~$C^1$-smooth function defining in $\IR^n$ a curve which does not lie entirely in an $(n-1)$ dimensional affine hyperplane. Let $Z\subset I$ be a~set of positive Lebesgue measure. If one of the conditions is fulfilled:
\begin{enumerate}[(i)]
\item the Lebesgue measure of $Z\cap(a,b)$ is positive for every nontrivial interval $(a,b)\subset I$,
\item the determinant
$$
d(x_1,\ldots,x_n)=\left|\begin{array}{ccc}
\varphi_1'(x_1) & \ldots & \varphi_1'(x_n)\\
\vdots & \ddots & \vdots\\
\varphi_n'(x_1) & \ldots & \varphi_n'(x_n)
\end{array}\right|
$$
is non-zero for almost every $(x_1,\dots,x_n)$ in $I^n$,
\end{enumerate}
then the image $\varphi(Z)$ belongs to the class $\mathcal{A}(\mathbb{R}^n)=\mathcal B(\IR^n)$.
\end{theorem}

In this paper we prove a topological counterpart of Ger's Theorem~\ref{Ger}.

A subset $A$ of a topological vector space $X$ is defined to be
\begin{itemize}
\item {\em flat} if the affine hull of $A$ is nowhere dense in $X$;
\item {\em nowhere flat} if each non-empty relatively open subset $U\subset A$ is not flat in $X$.
\end{itemize}

By a {\em continuum} we understand a connected compact metrizable space.

\begin{theorem}\label{t:main} Let $n\in\IN$. For any non-flat continuum $K\subset\IR^n$ its $n$-fold sum $K^{+n}$ has non-empty interior in $\IR^n$ and hence $K$ belongs to the class $\mathcal A(\IR^n)=\mathcal B(\IR^n)$. Moreover, if $K$ is locally connected and nowhere flat in $\IR^n$, then for any non-meager analytic subspace $A$ of $K$ the $2n$-fold sum $A^{+2n}$ has non-empty interior in $\IR^n$, which implies that $A\in\mathcal A(\IR^n)=\mathcal B(\IR^n)$.
\end{theorem}

\begin{remark} The first part of Theorem~\ref{t:main} answers a problem posed by the last author in \cite{Jablonski3}.
\end{remark}

In the proof of Theorem~\ref{t:main} we shall apply a non-trivial result of Holszty\'nski \cite{Hol} and Lifanov \cite{Lif} on the dimension properties of products of continua.  Let us recall that a closed subset $S$ of a topological space $X$ is called a {\em separator} between subsets $A$ and $B$ of $X$ if $A$ and $B$ are contained in distinct connected components of the complement $X\setminus S$.

The following proposition is due to Holszty\'nski \cite{Hol} and Lifanov \cite{Lif} and is discussed by Engelking in \cite[1.8.K]{End}.

\begin{proposition}[Holszty\'nski, Lifanov]\label{t:HL} Let $K_1,\dots,K_n$ be continua and $K:=\prod_{i=1}^nK_i$ be their product. For every positive integer $i\le n$ let $a_i^-,a_i^+$ be two distinct points in $K_i$ and let $S_i$ be a separator between the sets $A_i^-:=\{(x_k)_{k=1}^n\in K:x_i=a_i^-\}$ and $A_i^+:=\{(x_k)_{k=1}^n\in K:x_i=a_i^+\}$ in $K$. Then the intersection $\bigcap_{i=1}^nS_i$ is not empty.
\end{proposition}

The principal ingredient in the proof of Theorem~\ref{t:main} is the following result, which can have an independent value.

\begin{theorem}\label{main} Let $K_1,\dots,K_n$ be continua in $\mathbb R^n$ containing the origin of $\IR^n$. Assume that each continuum $K_i$ contains a point $e_i$ such that the vectors $e_1,\dots,e_n$ are linearly independent. Then the sum $K:=K_1+\dots+K_n$ has non-empty interior in $\IR^n$ and the Lebesgue measure of $K$ is not smaller than the volume of the parallelotope $P:=[0,1]\cdot e_1+\dots+[0,1]\cdot e_n$.\end{theorem}

\begin{proof} After a suitable linear transformation of $\IR^n$, we can assume that $e_1,\dots,e_n$ coincide with the standard basis of $\IR^n$, which means that $e_1=(1,0,\dots,0)$, $e_2=(0,1,0,\dots,0)$, \dots, $e_n=(0,\dots,0,1)$.
In this case we should prove that the sum $K=K_1+\dots+K_n$ has non-empty interior in $\IR^n$ and the Lebesgue measure $\lambda(K)$ of $K$ is not smaller than the volume $\lambda(P)=1$ of the cube $P=[0,1]^n$.

On the space $\IR^n$ we consider the sup-norm $\|x\|=\max_{1\le i\le n}|x_i|$.
Let $\delta:=\max\{\|x\|:x\in\bigcup_{i=1}^n K_i\}$. Choose numbers $s,l\in\IN$ such that $l>s+(n-1)\delta$. Moreover, if $\lambda(K)<1$, then we can replace $s$ and $l$ by larger numbers and additionally assume that $\big(\frac{2s}{2l+1})^n>\lambda(K)$.

For every positive integer $i\le n$, consider the finite set $Z_i=\{k\cdot e_i:k\in\IZ,\;|k|\le l\}$ in $\IR^n$ and observe that the sum $\tilde K_i:=K_i+Z_i$ is a continuum containing the set $Z_i$ (as $K_i$ contains zero). Let $Z:=Z_1+\dots+Z_n\subset \IZ^n\subset\IR^n$ and observe that $$K+Z=(K_1+\dots+K_n)+(Z_1+\dots+Z_n)=(K_1+Z_1)+\dots+(K_n+Z_n)=\tilde K_1+\dots+\tilde K_n.$$

\begin{claim}\label{claim} $[-s,s]^n\subset K+Z$.
\end{claim}

\begin{proof} To derive a contradiction, suppose that $[-s,s]^n\not\subset K+Z$ and fix a point $(y_i)_{i=1}^n\in [-s,s]^n\setminus (K+Z)$.
For every positive integer $k\le n$ denote by $\pr_k:\IR^n\to\IR$, $\pr_k:(x_i)_{i=1}^n\mapsto x_k$, the coordinate projection.

Also let $$Y_k:=\{x\in\IR^n:\pr_k(x)=y_k\},\;\;Y_k^-:=\{x\in\IR^n:\pr_k(x)<y_k\},\;\;Y_k^+:=\{x\in\IR^n:\pr_k(x)>y_k\}.$$

Consider the continuous map $\Sigma:\prod_{i=1}^n\tilde K_i\to K+Z$, $\Sigma:(x_i)_{i=1}^n\mapsto\sum_{i=1}^nx_i$. For every positive integer $k\le n$ let $p_k:\prod_{i=1}^n\tilde K_i\to\tilde K_k$, $p_k:(x_i)_{i=1}^n\mapsto x_k$, be the coordinate projection. The sets $F_k^-:=p_k^{-1}(-l\cdot e_k)$ and $F_k^+:=p_k^{-1}(l\cdot e_k)$ will be called the negative and positive $k$-faces of the ``cube'' $\tilde K:=\prod_{i=1}^n\tilde K_i$.

We claim that $\Sigma(F_k^{-})\subset Y_k^-$ and $\Sigma(F_k^+)\subset Y_k^+$.
Indeed, for any $\tilde x=(\tilde x_i)_{i=1}^n\in F_k^+\subset \tilde K$ we can find  sequences $(x_i)_{i=1}^{n}\in\prod_{i=1}^nK_i$ and $(z_i)_{i=1}^n\in\prod_{i=1}^nZ_i$ such that $(\tilde x_i)_{i=1}^n=(x_i+z_i)_{i=1}^n$.

Taking into account that $\pr_k(z_i)=0$ for all $i\ne k$, we conclude that
\begin{multline*}
\pr_k\circ\Sigma(\tilde x)=\sum_{i=1}^n\pr_k(\tilde x_i)=\pr_k(\tilde x_k)+\sum_{i\ne k}\pr_k(\tilde x_i)=\pr_k(l\cdot e_k)+\sum_{i\ne k}\pr_k(x_i+z_i)=\\
=l+\sum_{i\ne k}\pr_k(x_i)\ge l-\sum_{i\ne k}\|x_i\|\ge l-(n-1)\delta>s\ge y_k
\end{multline*}
and hence $\Sigma(\tilde x)\in Y_k^+$ and finally $\Sigma(F_k^+)\subset Y_k^+$.

By analogy we can prove that $\Sigma(F_k^-)\subset Y_k^-$. Then $\Sigma^{-1}(Y_k)$ is a separator between the $k$-faces $F_k^-$ and $F_k^+$ of the ``cube'' $\tilde K$.

Since $\bigcap_{k=1}^nY_k=(y_k)_{k=1}^n\notin K+Z=\Sigma(\tilde K)$, the intersection $\bigcap_{k=1}^n\Sigma^{-1}(Y_k)$ is empty, which contradicts Proposition~\ref{t:HL}.
\end{proof}

Now we continue the proof of Theorem~\ref{main}. By Claim~\ref{claim}, $[-s,s]^n\subset K+Z$. Taking into account that the set $Z$ is finite, we conclude that the set $K=K_1+\dots+K_n$ has non-empty interior in $\IR^n$. Moreover, $[-s,s]^n\subset K+Z$ implies $(2s)^n=\lambda([-s,s]^n)\le \lambda(K+Z)\le\lambda(K)\cdot|Z|=\lambda(K)\cdot(2l+1)^n$ and hence $\big(\frac{2s}{2l+1}\big)^n\le\lambda(K)$. Then $\lambda(K)\ge 1$ by the choice of the numbers $l$ and $s$.
\end{proof}

The following corollary of Theorem~\ref{main} implies the first part of Theorem~\ref{t:main}.

\begin{corollary}\label{c1} For a continuum $K\subset\IR^n$ and its $n$-fold sum $K^{+n}$ the following conditions are equivalent:
\begin{enumerate}
\item[\textup{(1)}] $K^{+n}$ has non-empty interior in $\IR^n$;
\item[\textup{(2)}] $K^{+n}$ has positive Lebesgue measure in $\IR^n$;
\item[\textup{(3)}] $K$ is not flat in $\IR^n$;
\item[\textup{(4)}] for any non-zero linear continuous functional $f:\IR^n\to \IR$ the image $f(K)$ has non-empty interior in $\IR$;
\item[\textup{(5)}] $K$ belongs to the family $\mathcal A(\IR^n)=\mathcal B(\IR^n)$;
\item[\textup{(6)}] $K$ belongs to the family $\mathcal C(\IR^n)$.
\end{enumerate}
\end{corollary}

\begin{proof} The implications $(1)\Ra(2)\Ra(3)$ are trivial. To prove that  $(3)\Ra(1)$, assume that $K$ is not flat in $\IR^n$. After a suitable shift, we can assume that $K$ contains the origin of the vector space $\IR^n$.

Let $E\subset K$ be a maximal linearly independent subset of $K$. Assuming that $|E|<n$ we would conclude that $K$ is contained in the linear hull of the set $E$ and hence is flat. So, $|E|=n$ and we can write $E$ as $E=\{e_1,\dots,e_n\}\subset K$. Applying Theorem~\ref{main} to the continua $K_i:=K$, $1\le i\le n$, we conclude that the sum $K^{+n}=K_1+\dots+K_n$ has non-empty interior in $\IR^n$.

The equivalence $(3)\Leftrightarrow(4)$ is proved in Lemma~\ref{l:flat} below and
the implication $(1)\Ra(5)$ follows from Proposition~\ref{tx} and $(5)\Ra(6)$ is trivial. To finish the proof, it suffices to show that $(6)\Ra(3)$, which is equivalent to $\neg(3)\Ra\neg(6)$.
So, assume that the continuum $K$ is flat in $\IR^n$. Then $K\subset x+L$ for some $x\in \IR^n$ and some linear subspace $L\subset\IR^n$ of dimension $\dim(L)=n-1$. Let $\IR^n/L$ be the quotient space and $q:\IR^n\to \IR^n/L$ be the quotient linear operator.

Since the quotient space $\IR^n/L$ is topologically isomorphic to $\IR$, it admits a discontinuous additive function $a:\IR^n/L\to\IR$. Then $f=a\circ q:\IR^n\to\IR$ is a discontinuous additive function such that $f(K)\subset f(x+L)=\{f(x)\}$, which means that $K\notin\mathcal C(\IR^n)$.
\end{proof}

\begin{lemma}\label{l:flat} A connected subset $K$ of a locally convex topological vector space $X$ is not flat if and only if for any non-zero linear continuous functional $f:X\to\IR$ the image $f(K)$ has non-empty interior in $X$.
\end{lemma}

\begin{proof} If $A$ is flat, then the affine hull of $A$ in $X$ is nowhere dense and hence $A\subset a+L$ for some $a\in A$ and some nowhere dense closed linear subspace of $X$. Using the Hahn-Banach Theorem,  choose a non-zero linear continuous functional $f:X\to\IR$ such that $L\subset f^{-1}(0)$. Then the image $f(A)\subset f(L+a)=\{f(a)\}$ is a~singleton (which has empty interior in the real line).

Now assuming that $A$ is not flat, we shall prove that for any non-zero linear continuous functional $f:X\to\IR$ the image $f(A)$ has non-empty interior in $\IR$. Since $A$ is not flat, its affine hull is dense in $X$ and hence the affine hull of $f(A)$ is dense in $\IR$. This implies that $f(A)$ contains two distinct points $a<b$. By the connectedness of $f(A)$ (which follows from the connectedness of $A$ and the continuity of $f$), the set $f(A)$ contains the interval $[a,b]$ and hence has non-empty interior in $\IR$.
\end{proof}

\begin{problem} Is there a compact subset $K\subset \IR^2$ such that $K+K$ has empty interior in $\IR^2$ but for any non-zero linear continuous functional $f:X\to\IR$ the image $f(K)$ has non-empty interior in $\IR$?
\end{problem}



Subsets $A_1,\dots,A_n$ of $\IR^n$ will be called {\em collectively nowhere flat in $\IR^n$} if  any non-empty relatively open subsets $U_1\subset A_1$, \dots, $U_n\subset A_n$ contain points $a_1,b_1\in U_1$, \dots, $a_n,b_n\in U_n$ such that the vectors $b_1-a_1,\dots,b_n-a_n$ form a basis of the linear space $\IR^n$.

For example, for a basis $e_1,\dots,e_n$ of $\IR^n$ with $n\ge 2$ the closed intervals $[0,1]\cdot e_1,\dots[0,1]\cdot e_n$ are collectively nowhere flat; yet each set $[0,1]\cdot e_i$ separately is flat.

It is easy to see that a subspace $A\subset\IR^n$ is nowhere flat in $\IR^n$ if and only if the sequence of $n$ its copies $A_1=A,\dots,A_n=A$ is collectively nowhere flat in $\IR^n$.

\begin{theorem}\label{t:main2} Let $K_1,\dots,K_n$ are collectively nowhere flat locally connected subspaces of $\IR^n$. For every non-meager subsets $B_1,\dots,B_n$ in $K_1,\dots,K_n$ the algebraic sum $B_1+\dots+B_n$ is non-meager in $\IR^n$.
\end{theorem}

\begin{proof}  To derive a contradiction, assume that the sum $B_1+\dots+B_n$ is meager and hence is contained in the countable union $\bigcup_{i\in\w}F_i$ of closed nowhere dense subsets of $\IR^n$. Consider the continuous map $\Sigma:
(\IR^n)^n\to\IR^n$, $\Sigma:(x_k)_{k=1}^n\mapsto\sum_{k=1}^nx_k$.  Taking into account that for every $i\le n$ the subset $B_i$ is not meager in $K_i$, we can apply a classical result of Banach \cite[\S 10.V]{Ku} and find a non-empty open set $W_i\subset K_i$  such that the intersection $W_i\cap B_i$ is a dense Baire subspace of $W_i$. Replacing $W_i$ by a smaller open subset of $W_i$, we can assume that the set $W_i$ is bounded in $\IR^n$. Replacing $B_i$ by $B_i\cap W_i$, we can assume that $B_i$ is a dense Baire space in $W_i$.

By \cite[8.44]{Kechris}, the product $\prod_{k=1}^n B_k$ of second countable Baire spaces $B_k$ is Baire. Since $\prod_{k=1}^n B_k\subset\bigcup_{i\in\w}\Sigma^{-1}(F_i)$, we can apply Baire Theorem and find $i\in\w$ such that the set $\Sigma^{-1}(F_i)\cap\prod_{k=1}^nB_k$ has non-empty interior in $\prod_{k=1}^n B_k$. Then we can find non-empty open sets $V_1\subset W_1,\dots,V_n\subset W_n$ such that $\prod_{k=1}^n(B_k\cap V_k)\subset\Sigma^{-1}(F_i)$. Since the spaces $K_1,\dots,K_n$ are locally connected, we can additionally assume that each set $V_k$ is connected and hence has compact connected closure $\overline{V}_k$ in $\IR^n$. The set $\Sigma^{-1}(F_i)$ is closed and hence contains the closure $\prod_{k=1}^n \overline{V}_k$ of the set $\prod_{k=1}^n(V_k\cap B_k)$ in $(\IR^n)^n$. Then the set $\overline{V}_1+\dots+\overline{V}_n=\Sigma(\prod_{k=1}^n\overline{V}_k)\subset F_i$ is nowhere dense in $\IR^n$.

On the other hand, taking into account that the sets $K_1,\dots,K_n$ are collectively nowhere flat in $\IR^n$, in each set $\overline{V}_k$ we can choose two points $a_k,b_k$ such that the vectors $b_1-a_1,\dots,b_n-a_n$ form a basis of the vector space $\IR^n$. Applying Theorem~\ref{main}, we can conclude that the set $\overline{V}_1+\dots+\overline{V}_n$ has non-empty interior and hence cannot be contained in the nowhere dense set $F_i$. This contradiction completes the proof.
\end{proof}

The following corollary of Theorem~\ref{t:main2} yields the second part of Theorem~\ref{t:main}.

\begin{corollary}\label{c:A} Let $K$ be a nowhere flat locally connected subset of $\IR^n$ and $A$ be a non-meager analytic subspace of $K$. Then the $n$-fold sum $A^{+n}$ of $A$ is a non-meager analytic subset of $\IR^n$, the $2n$-fold sum  $A^{+2n}$ has non-empty interior and the set $(A-A)^{+n}$
 is a~neighborhood of zero in $\IR^n$. Moreover, the set $A$ belongs to the families $\mathcal A(\IR^n)=\mathcal B(\IR^n)\subset\mathcal C(\IR^n)$.
 \end{corollary}

 \begin{proof} Since the set $A$ is non-meager in $K$, by \cite[\S 10.V]{Ku}, there exists a non-empty open set $V\subset K$ such that $V\cap A$ is a~dense Baire subspace of $V$. Since $K$ is locally connected, we can assume that $V$ is connected and so is its closure $\overline{V}$ in $K$. Replacing $K$ by $\overline{V}$ and $A$ by $A\cap V$, we can assume that $A$ is a dense Baire subspace of $K$. By Theorem~\ref{t:main2}, the $n$-fold sum $A^{+n}$ of $A$ is a non-meager subset of $\IR^n$. The subspace $A^{+n}$ is analytic, being a continuous image of the analytic space $A^n$. Applying Pettis-Piccard Theorem \cite[Corollary 5]{Pet}, \cite{Pic} (also~\cite[Theorem~1]{Kominek}), we can conclude that the sum $A^{+n}+A^{+n}=A^{+2n}$ has non-empty interior and the difference $A^{+n}-A^{+n}$ is a~neighborhood of zero in $\IR^n$.

Since $A^{+2n}$ has non-empty interior in $\IR^n$, we can apply Proposition~\ref{tx} and conclude that $A\in\mathcal A(\IR^n)=\mathcal B(\IR^n)\subset\mathcal C(\IR^n)$.
 \end{proof}

 Writing down the definition of the class $\mathcal A(\IR^n)$ and applying Corollaries~\ref{c1} and \ref{c:A}, we obtain the following characterization.

\begin{corollary}\label{c4} For a mid-convex function $f:D\to\IR$ defined on a non-empty open convex set $D\subset\IR^n$, the following conditions are  equivalent:
\begin{enumerate}
\item[\textup{(1)}] $f$ is continuous;
\item[\textup{(2)}] $f$ is upper bounded on some non-flat continuum $K\subset D$;
\item[\textup{(3)}] $f$ is upper bounded on some non-meager analytic subspace of a nowhere flat locally connected subset of~$D$.
\end{enumerate}
\end{corollary}

Now we present an example showing that the condition of (collective) nowhere flatness in Corollaries~\ref{c:A}, \ref{c4} (and Theorem~\ref{t:main2}) is essential.

\begin{example} Let $C=\big\{\sum_{n=1}^\infty \frac{2x_n}{3^n}:(x_n)_{n=1}^\infty\in\{0,1\}^{\IN}\big\}$ be the standard Cantor set in the interval $[0,1]$.
Let $f:C\to[0,1]$ be the continuous map assigning to each point $\sum_{n=1}^\infty \frac{2x_n}{3^n}$ of the Cantor set $C$ the real number $\sum_{n=1}^\infty\frac1{2^n}$. Let $\bar f:[0,1]\to[0,1]$ be the unique monotone function extending $f$. The function $\bar f$ is known as {\em Cantor ladders}. It is uniquely defined by the condition $\bar f^{-1}(y)=\mathrm{conv}(f^{-1}(y))$ for $y\in[0,1]$. Let $\Gamma_f:=\{(x,f(x)):x\in C\}$ and $\Gamma_{\bar f}:=\{(x,\bar f(x)):x\in [0,1]\}$ be the graphs of the functions $f$ and $\bar f$. The set $\Gamma_f$ is nowhere flat and zero-dimensional, and the set $\Gamma_{\bar f}$ is connected but not nowhere flat in the plane $\IR\times\IR$.

It is easy to see that $A:=\Gamma_{\bar f}\setminus\Gamma_f$ is an open dense subset of $\Gamma_{\bar f}$. By Corollary~\ref{c1} (see also the proof of Theorem 9.5.2 in \cite{Kuczma2}), the sum $S:=\Gamma_{\bar f}+\Gamma_{\bar f}$ has non-empty interior in the plane $\IR^2$. On the other hand, the sum $A+A$ is meager in $\mathbb R^2$. Moreover, since $A+A$ is contained in the union of countably many parallel lines in $\IR\times\IR$, the $\IQ$-linear hull of $A$ has uncountable codimension in $\IR^2$, which allows us to construct a discontinuous additive function $a:\IR^2\to\IR$ such that $a(A)=\{0\}$, witnessing that $A\notin\mathcal C(\IR^2)$.
\end{example}







\end{document}